\documentclass[11pt]{amsart}

\usepackage{amsmath,amssymb,graphicx}
\newtheorem{theorem}{Theorem}[section]
\newtheorem{lemma}[theorem]{Lemma}
\newtheorem{corollary}[theorem]{Corollary}

\newtheorem{sublemma}{}[theorem]

\newcommand{\ba}{\backslash}
\newcommand{\cl}{{\rm cl}}
\newcommand{\fcl}{{\rm fcl}}

\newcommand{\si}{{\rm si}}

\newcommand{\thc}{$3$-connected}
\newcommand{\ths}{$3$-separation}
\newcommand{\ifc}{internally $4$-connected}

\newcommand{\cn}{contradiction}
\newcommand{\btu}{\bigtriangleup}
\newcommand{\ftv}{$(4,3)$-violator}

\newcommand{\ort}{orthogonality}

\newcommand{\ns}{non-sequential}
\newcommand{\al}{\alpha}

\begin{document}

\title[Every element in three triangles]{Internally $4$-connected binary matroids with every element in three triangles}

\author{Carolyn Chun}
\address{United States Naval Academy, Annapolis, Maryland, USA}
\email{chun@usna.edu}

\author{James Oxley}
\address{Department of
 Mathematics, Louisiana State University, Baton Rouge, Louisiana, USA}
\email{oxley@math.lsu.edu}

\date{\today}

\begin{abstract} 
Let $M$ be an \ifc\ binary matroid with every element in three triangles. Then  $M$ has at least four elements   $e$ such that $\text{si}(M/e)$ is \ifc.
\end{abstract}

\maketitle

\section{Introduction}

%This definition from~\cite{wtt} has the attractive property that $(X,E-X)$ is $k$-separating in $M$ exactly when it is $k$-separating in $M^*$.  
%Thus a matroid is $n$-connected if and only if its dual is. 

%Moreover, this matroid definition of $n$-connectivity is relatively compatible with the graph notion of $n$-connectivity when $n$ is $2$ or $3$. For example, when $G$ is a graph with at least four vertices and with no isolated vertices, $M(G)$ is a $3$-connected matroid if and only if $G$ is a $3$-connected simple graph. But the link between $n$-connectivity for matroids and graphs breaks down for $n \ge 4$. In particular, a $4$-connected matroid with at least six elements cannot have a triangle. Hence, for $r \ge 3$, neither $M(K_{r+1})$ nor $PG(r-1,2)$ is $4$-connected. This motivates the consideration of other types of $4$-connectivity that admit certain $3$-separations.

Terminology in this note will follow~\cite{oxrox}.   A matroid is \emph{internally $4$-connected}  if it is $3$-connected and, for every $3$-separation $(X,Y)$ of $M$, either $X$ or $Y$ is a triangle or a triad of $M$.  

The purpose of this note is to prove the following technical result. 

\begin{theorem}
\label{bigun}
Let $M$ be a binary \ifc\ matroid in which every element is in exactly three triangles. Then $M$ has at least four elements $e$ such that $si(M/e)$ is \ifc. Morever, if $M$ has fewer than six such elements, then these elements are in a $4$-element cocircuit.
\end{theorem}

\section{Preliminaries}

%A graph $G$ without isolated vertices is {\it internally $4$-connected} if $M(G)$ is internally $4$-connected. 

This section introduces some  basic material relating to matroid connectivity. 
For a matroid $M$, let $E$ be the ground set of $M$ and $r$ be its rank function.   
The {\it connectivity function} $\lambda_M$ of $M$ is defined on all subsets $X$ of $E$ by $\lambda_M(X) = r(X) + r(E-X) - r(M)$.  Equivalently, $\lambda_M(X) = r(X) + r^*(X) - |X|.$ We will sometimes abbreviate $\lambda_M$ as $\lambda$.
For a positive integer $k$, a subset $X$ or a partition $(X,E-X)$ of $E$ is 
{\em $k$-separating} if $\lambda_M(X)\le k-1$.  A $k$-separating partition $(X,E-X)$ of $E$ is a {\it $k$-separation} if  $|X|,|E-X|\ge k$. If $n$ is an integer exceeding one, a matroid  is {\it $n$-connected} if it has  no $k$-separations for all $k < n$.  
 Let $(X,Y)$ be a \ths\ in a matroid $M$.  If $|X|,|Y|\geq 4$, then we call $X,Y$, or $(X,Y)$ a \emph{$(4,3)$-violator}  since it certifies that $M$ is not \ifc.  For example, if $X$ is a \emph{$4$-fan}, that is, a $4$-element set containing a triangle and a triad, then $X$ is a \ftv\ provided $|Y|\geq 4$.  
 
 %We usually order the elements of a $4$-fan $(\al ,\be,\ga,\de)$ when $\{\al,\be,\ga\}$ is a triangle and $\{\be,\ga,\de\}$ is a triad.  
%If, in addition,  $\{\ga,\de,\varepsilon\}$ is a triangle, then $(\al,\be,\ga,\de,\varepsilon)$ is a $5$-fan.    

%Note that, if $U$ is coindependent, then it is in the complement of a basis of $M$.  
%Hence $U\subseteq \cl (E(M)-U)$.  
In a matroid $M$, a set $U$  is {\em fully closed} if it is closed in both $M$ and $M^*$. 
The {\it full closure} $\fcl(Z)$ of a set $Z$ in $M$ is the intersection of all fully closed sets containing $Z$.  
The full closure of $Z$ may be obtained by alternating between taking the closure and the coclosure until both operations leave the set unchanged.   
Let $(X,Y)$ be a partition of $E(M)$. If $(X,Y)$ is $k$-separating in $M$ for some positive integer $k$,  and $y$ is an element of  $Y$ that is also in $\cl(X)$ or $\cl ^*(X)$, then it is well known and easily checked that $(X \cup y,Y-y)$ is $k$-separating, and we say that we have {\it moved} $y$ into $X$. More generally, $(\fcl(X),Y-\fcl(X))$ is $k$-separating in $M$. 

The following elementary result will be used repeatedly.

\begin{lemma}
\label{3conn}
If $M$ is an \ifc\ binary matroid and $e\in E(M)$, then $\si (M/e)$ is \thc.  
\end{lemma}

\begin{proof}

The result is easily checked if $|E(M)| < 4$, so we may assume that $|E(M)| \ge 4$. Since $M$ is $3$-connected and binary, $|E(M)| \ge 6$ and both $M/e$ and $\si(M/e)$ are $2$-connected. If $|E(M)| \in \{6,7\}$, then $M$ is isomorphic to $M(K_4), F_7$, or $F_7^*$ and again the result is easily checked. Thus we may assume that $|E(M)| \ge 8$. 

Now let $M' = \si(M/e)$ and suppose that $M'$ has a $2$-separation $(X,Y)$. We may assume that $|X| \ge |Y|$. Suppose $|Y| = 2$. Then $Y$ is a $2$-cocircuit $\{y_1,y_2\}$ of $M'$. As $\{y_1,y_2\}$ is not a $2$-cocircuit of $M/e$ and $M$ is binary, we see that, in $M/e$, either one or both of $y_1$ and $y_2$ is in a $2$-element parallel class. Thus we may assume that $M/e$ has $\{y_1,y'_1\}$ as a circuit and $\{y_1,y'_1,y_2\}$ as a cocircuit, or $M/e$ has $\{y_1,y'_1\}$ and $\{y_2,y'_2\}$ as circuits and has $\{y_1,y'_1,y_2,y'_2\}$ as a cocircuit. Hence $M$ has $\{e,y_1,y_1',y_2\}$ as a $4$-fan or has $\{y_1,y'_1,y_2,y'_2\}$ as both a circuit and a cocircuit. Since $|E(M)| \ge 8$, each possibility contradicts the fact that $M$ is \ifc. We conclude that $|Y| \ge 3$. 

Let $(X',Y')$ be obtained from $(X,Y)$ by adjoining each element of $E(M/e) - E(M')$ to the side of $(X,Y)$ that contains an element parallel to it. Then $r_{M/e}(X') = r_{M'}(X)$ and $r_{M/e}(Y') = r_{M'}(Y)$, so $(X',Y')$ is a $2$-separation of $M/e$. Hence $(X', Y' \cup e)$ and $(X' \cup e, Y')$ are $3$-separations of $M$. As $|Y' \cup e| \ge 4$ and $|E(M) \ge 8$, this gives a \cn.
\end{proof}

Let $n$ be an   integer exceeding one. If $M$ is $n$-connected, an $n$-separation $(U,V)$ of $M$ is {\it sequential} if $\fcl(U)$ or $\fcl(V)$ is $E(M)$. In particular, when $\fcl(U) = E(M)$, there is an ordering $(v_1,v_2,\ldots,v_m)$ of the elements of $V$ such that 
$U \cup \{v_m,v_{m-1},\ldots,v_i\}$ is $n$-separating for all $i$ in $\{1,2,\ldots,m\}$. When this occurs, the set $V$ is called {\it sequential}. 

\section{Small matroids}

We begin this section by noting two useful results.

\begin{lemma}
\label{dracula} 
Let $M$ be a matroid in which every element is in exactly three triangles. Then $M$ has exactly $|E(M)|$ triangles.
\end{lemma}

\begin{proof}
Consider the set of ordered pairs $(e,T)$ where $e \in E(M)$ and $T$ is a triangle of $M$ containing $e$. The number of such pairs is $3|E(M)|$ since each element is in exactly three triangles. As each triangle contains exactly three elements, this number is also three times the number of triangles of $M$.
\end{proof}

\begin{lemma}
\label{noodd}
Let $M$ be an \ifc\ binary matroid in which every element is in exactly three triangles. Then $M$ has   no cocircuits of  odd size.  
\end{lemma}

\begin{proof}
For a cocircuit $C^*$ of $M$, we construct an auxiliary graph $G$ as follows.  
Let $C^*$ be the vertex set of $G$, and let $c_1c_2$ be an edge exactly when $c_1$ and $c_2$ are members of $C^*$ that are contained in a triangle of $M$.  
Since every element in is three triangles of $M$, every vertex in $G$ has degree three by \ort\ and the fact that $M$ is binary.  
%The sum of the degrees of the vertices in $G$ is twice the number of edges.  
Hence $|C^*|$, which equals the number of vertices of $G$ with odd degree,  is even.  
\end{proof}

To prove the next lemma, we shall use the  following theorem of Qin and Zhou \cite{qz}.

\begin{theorem}
\label{qzthm}
Let $M$ be an \ifc\ binary matroid with no minor isomorphic to any of $M(K_{3,3})$, $M^*(K_{3,3})$, $M(K_5)$, or $M^*(K_5)$. Then either $M$ is isomorphic to the cycle matroid of a planar graph, or $M$ is isomorphic to $F_7$ or $F^*_7$.
\end{theorem}

\begin{lemma}
\label{smalltime}
Let $M$ be an \ifc\ binary matroid in which every element is in exactly three triangles and $|E(M)| \le 13$. Then $M$ is isomorphic to $F_7$ or $M(K_5)$. Hence $\si(M/e)$ is \ifc\ for all elements $e$ of $M$.
\end{lemma}

\begin{proof} 
Assume that $M$ is not isomorphic to $F_7$ or $M(K_5)$. Suppose first that $M$ has none of $M(K_{3,3})$, $M^*(K_{3,3})$, $M(K_5)$, or $M^*(K_5)$ as a minor. As $F^*_7$ has no triangles, it follows that $M$ is isomorphic to the cycle matroid of a planar graph $G$. As every edge of $G$ is in exactly three triangles, but $M(G)$ is \ifc, every vertex has degree at least four. Hence $|E(G)| \ge 2|V(G)|$. Moreover, by Lemma~\ref{noodd}, every vertex of $G$ has even degree. Clearly $|V(G)| \neq 4$. Moreover, $|V(G)| \neq 5$, otherwise $M\cong M(K_5)$; a \cn. As $|E(G)| \le 13$, it follows that $|V(G)| = 6$ and $|E(G)| = 12$. Then $G$ is obtained from $K_6$ by deleting the edges of a perfect matching. But no edge of this graph is in exactly three triangles.

We may now assume that $M$ has an $N$-minor for some $N$ in $\{M(K_{3,3}),  M^*(K_{3,3}), M(K_5), M^*(K_5)\}$. By the Splitter Theorem for $3$-connected matroids, there is a sequence $M_0,M_1, \dots,M_k$ of $3$-connected matroids such that $M_0 \cong N$ and $M_k \cong M$, while $|E(M_{i+1}) - E(M_i)| = 1$ for all $i$ in $\{0,1,\ldots,k-1\}$. Since $|E(M) \ge 9$ and $|E(M)| \le 13$, it follows that $k \in \{0,1,2,3,4\}$. 

Suppose that some $M_i$ is obtained from its successor by contracting an element $e$. Then $M/e$ has an $N$-minor. But $\si(M/e)$ has at most nine elements. Thus $|E(M)| = 13$ and $N$ is $M(K_{3,3})$ or  $M^*(K_{3,3})$. Since $\si(M/e)$ must contain triangles, $N$ is   $M^*(K_{3,3})$. Now, by Lemma~\ref{noodd}, every cocircuit of  $M/e$ is even.  Moreover, $M/e$ has exactly three 2-circuits. The union of these three 2-circuits cannot have rank two in $M/e$ otherwise $M$ has $F_7$ as a restriction but the remaining six elements of $M$ cannot all be in exactly three triangles of $M$. 
Let $a, b$ and $c$ be the three elements of $M^*(K_{3,3})$ that are in $2$-circuits in $M/e$. Then one easily checks that there are two intersecting triangles of $M^*(K_{3,3})$ whose union  contains exactly two elements of $\{a,b,c\}$. The  cocircuit of $M/e$ whose complement is the union of the closure of these  two triangles is odd; a \cn.

We now know that $M$ is an extension of $N$ by at most four elements. Let $N  = M\ba D$. Then $|D| \ge 1$ so $|E(M)| \ge 10$. Moreover, 
$N$ has at least $|E(M)| - 3|D|$ triangles. It is straightforward to check that the last number is  positive, so $N$ cannot be $M(K_{3,3})$ or $M^*(K_5)$.  Thus $N$ is $M^*(K_{3,3})$ or $M(K_5)$. Each element of $M(K_5)$ is in three triangles, so $N \neq M(K_5)$ since each element of $E(M) - E(N)$ must be in a triangle with some element of $M(K_5)$; a \cn. We deduce that $N = M^*(K_{3,3})$. Now  $M^*(K_{3,3})$ has exactly six triangles with each element being in precisely two triangles. Thus, in $M$, there are six triangles each containing a single element of $M^*(K_{3,3})$ and two elements of $E(M) - E(N)$. As $|E(M)| - E(N)| \le 4$, there are at most six triangles containing exactly two elements of $E(M) - E(N)$. We deduce that $|E(M)| = 13$ so $M$ can be obtained from $PG(3,2)$ by deleting exactly two elements. As $PG(3,2)$ has exactly seven triangles containing each element, deleting two elements leaves each element in at least five triangles; a \cn. 
\end{proof}

\section{Small cocircuits}

In this section, we move towards proving the main result by dealing with 4-cocircuits and certain special 6-cocircuits in $M$. 
Throughout the section, we will assume that $M$ is an \ifc\ binary matroid  in which every element is in exactly three triangles, and $|E(M)| \ge 14$.

\begin{lemma}
\label{no5}
If $C^*$ is a $4$-element cocircuit of $M$, then, for all $e$ in $C^*$, the matroid $\si(M/e)$ is \ifc\ having no triads.
\end{lemma}

\begin{proof}
Suppose that  $C^*=\{e,f_1,f_2,f_3\}$ and $\si(M/e)$ is not \ifc. 
As $M$ is \ifc, $r(C^*)=4$.  
As $e$ is in three triangles of $M$, there are elements $\{g_1,g_2,g_3\}$ such that $\{e,f_i,g_i\}$ is a triangle for all $i$.  
%The seven elements in $\{e,f_1,f_2,f_3,g_1,g_2,g_3\}$ are not coplanar, since otherwise $\lambda (C^*)=r(C^*)+r^*(C^*)-|C^*|\leq 3+3-4=2$; a \cn.  
As $f_i$ is in three triangles for all $i$, by \ort\ and the fact that $M$ is binary, there are elements $\{h_1,h_2,h_3\}$ such that $\{f_1,f_2,h_1\},\{f_1,f_3,h_3\}$, and $\{f_2,f_3,h_2\}$ are triangles.  
This forces $\{g_1,g_2,h_1\},\{g_1,g_3,h_3\}$, and $\{g_2,g_3,h_2\}$ to be triangles, so $g_i$ is in no other triangle of $M$ for all $i$.  

Let $M'=\si (M/e)=M/e\ba f_1,f_2,f_3$.  
Lemma~\ref{3conn} implies that $M'$ is \thc.  
The set $\{g_1,g_2,g_3,h_1,h_2,h_3\}$ forms an $M(K_4)$-restriction in $M'$.  
Suppose $M'$ has a \ns\ \ths. Then we may assume that $\{g_1,g_2,g_3,h_1,h_2,h_3\}$ is contained in one side of the \ths.  
Since $\{f_i,g_i\}$ is a circuit in $M/e$, we may add $f_1,f_2$, and $f_3$ to the side containing the $M(K_4)$-restriction, and then add $e$ to get a \ftv\ of $M$; a \cn.  
We deduce that a \ftv\ of $\si (M/e)$ is a sequential $3$-separation.  

We show next that 

\begin{sublemma}
\label{notriads}
$M/e\ba f_1,f_2,f_3$ has no triads. 
\end{sublemma}

Suppose $M/e\ba f_1,f_2,f_3$ has a triad $\{\beta,\gamma,\delta\}$. Then $M\ba f_1,f_2,f_3$ has  $\{\beta,\gamma,\delta\}$ as a cocircuit. By Lemma~\ref{noodd}, we may assume that $\{\beta,\gamma,\delta,f_1,f_2,f_3\}$ or $\{\beta,\gamma,\delta,f_1\}$ is a cocircuit of $M$. By \ort, in the first case, 
$\{\beta,\gamma,\delta\} = \{g_1,g_2,g_3\}$ while, in the second case, $g_1 \in  \{\beta,\gamma,\delta\}$. In the first case, let $Z = \{e,f_1,f_2,f_3,g_1,g_2,g_3\}$. Then $r(Z) \le 4$ while $|Z| - r^*(Z) \ge 2$, so $\lambda(Z) \le 2$, a \cn\ as $|E(M)| \ge 14$. 

In the second case, $M$ has a $4$-cocircuit $D^*$ such that $C^* \cap D^* = \{f_1\}$ and $g_1 \in D^*$. Apart from $\{f_1,e,g_1\}$, the other triangles containing $f_1$ must meet $C^* - \{f_1,e\}$ in distinct elements and must meet $D^* - \{f_1,g_1\}$ in distinct elements. Thus $r(C^* \cup D^*) \le 4$ and $|C^* \cup D^*| - r^*(C^* \cup D^*) \ge 2$, so $\lambda(C^* \cup D^*) \le 2$; a \cn\ since $|E(M)| \ge 14$. Thus \ref{notriads} holds.

By \ref{notriads}, $M/e\ba f_1,f_2,f_3$ has no $4$-fans and so has no sequential $3$-separation that is a \ftv. This \cn\ completes the proof. 
\end{proof}

\begin{lemma}
\label{no6}
Take $e\in E(M)$ and the three triangles $T_1,T_2,$ and $T_3$ containing $e$.  
If $(T_1\cup T_2\cup T_3)-e$ is a cocircuit $C^*$, then $\si (M/x)$ is \ifc\ for every  element $x$ of $C^*$.  
\end{lemma}

\begin{proof}
 Let $T_i=\{e,f_i,g_i\}$ for each $i\in\{1,2,3\}$.  
Note that $T_1,T_2$, and $T_3$ are not coplanar,  otherwise their union forms an $F_7$-restriction, and $C^*$ contains a triangle; a \cn\ to the fact that $M$ is binary.  Suppose the lemma fails. Then we may assume that $\si(M/f_3)$ is not \ifc. 

As $f_1$ is in two triangles other than $T_1$, \ort\ and the fact that $M$ is binary imply that each  of these triangles contains an element of $\{f_2,g_2,f_3,g_3\}$.  
If $\{f_1,f_2\}$ and $\{f_1,g_2\}$ are each contained in a triangle, then the plane containing $T_1$ and $T_2$ is an $F_7$-restriction, so $e$ is in a fourth triangle; a \cn.  
Hence $f_1$ is in a single triangle with an element of $\{f_2,g_2\}$ and a single triangle with an element of $\{f_3,g_3\}$.  
Without loss of generality, $\{f_1,g_2,x_1\}$ and $\{f_1,g_3,x_2\}$ are triangles.  
By taking the symmetric difference of these triangles with the circuits $\{f_1,g_1,f_2,g_2\}$ and $\{f_1,g_1,f_3,g_3\}$, respectively, we see that $\{g_1,f_2,x_1\}$ and $\{g_1,f_3,x_2\}$ are also triangles.  
We have now identified all three of the triangles containing each element in $\{f_1,g_1\}$. But,  for each element in $\{f_2,g_2,f_3,g_3\}$, one of the triangles containing the element remains undetermined.
  
Either $\{f_2,g_3,x_3\}$ and $\{g_2,f_3,x_3\}$ are triangles, or  $\{f_2,f_3,y_3\}$ and $\{g_2,g_3,y_3\}$ are triangles.  
In each of these cases, we will  obtain the \cn\  that $\si (M/f_3)$ is \ifc.  
By Lemma~\ref{3conn}, $M'=\si (M/f_3)$ is \thc.  
Take $(U,V)$ to be  a \ftv\ in $M'$. 

Let $X = \{e,f_1,f_2,g_1,g_2,x_1\}$. Clearly the restriction of $M/f_3$ to $X$ is  isomorphic to $M(K_4)$. We may assume that $M' = M/f_3\ba Y$ where $Y$ is $\{g_3,x_2,x_3\}$ or $\{g_3,x_2,y_3\}$ depending on whether $\{f_3,g_2,x_3\}$ or $\{f_3,f_2,y_3\}$ is a triangle of $M$. Without loss of generality, we may also assume that $U$ spans $X$ in $M'$. Then $(U \cup X, V- X)$ is $3$-separating in $M'$ and it follows that $(U \cup X \cup Y \cup f_3, V - X)$ is 3-separating in $M$. Since $M$ has no \ftv, we deduce that $V$ is a sequential $3$-separating set in $M'$. Thus $M'$ has a triad 
$\{\beta, \gamma, \delta\}$.  By Lemma~\ref{noodd}, $M$ has a cocircuit $D^*$ where $D^*$ is $\{\beta, \gamma, \delta\} \cup Y$ or $\{\beta, \gamma, \delta\} \cup y$ for some $y$ in $Y$. In the first case, by \ort, $\{\beta, \gamma, \delta\} \subseteq X$. The last inclusion also follows by \ort\ in the second case since  $\{\beta, \gamma, \delta\}$ must meet $X$ and $M|X \cong M(K_4)$. Hence $X \cup Y \cup f_3$ contains at least two cocircuits. 
Since $r(X \cup Y \cup f_3) = 4$, it follows that $\lambda(X \cup Y \cup f_3) \le 2$; a \cn\ as $|E(M)| \ge 14$.
\end{proof}

\begin{lemma}
\label{seqsep}
Let $(X,Y)$ be an exact $4$-separation in $M$ with $X \subseteq \fcl(Y)$. If $M$ has no $4$-cocircuits, then 
 $X$ is coindependent, $r(X) = 3$, and $X \subseteq \cl(Y)$.
\end{lemma}
  
\begin{proof}
%Without loss of generality, suppose that $X\subseteq \text{fcl}(Y)$.  
If $X\subseteq \cl(Y)$, then $Y$ contains a basis of $M$, and $X$ is coindependent.  
As $r(X)+r^*(X)-|X|\leq 3$, the rank of $X$ is at most three, and the result holds.  
If $X\subseteq\cl ^*(Y)$, then $X$ is independent, so   
 $r^*(X)=3$.  
As $|X| \ge 4$, it follows that $X$ is a $4$-cocircuit; a contradiction.  

Beginning with $Y$, look at $\cl (Y),\cl ^*(\cl (Y)),\cl (\cl ^*(\cl(Y))),\dots$ until the first time we get $E(M)$.  
Consider the set $Y'$ that occurs before $E(M)$ in this sequence, let $X'=E(M)-Y'$, and let $e$ be the last element that was added in taking the closure or coclosure that equals $Y'$.  
Then either $Y'$ is a hyperplane and $X'$ is a cocircuit, or $Y'$ is a cohyperplane and $X'$ is a circuit.  

Suppose $X'$ is a circuit.  
As $r(X')+r^*(X')-|X'|\leq 3$, we see that $r^*(X')\leq 4$.  
Thus, as  $X'$ does not contain a $4$-cocircuit, it     
 is coindependent, so it has size at most four.  
We may assume that $X'\subsetneqq X$, otherwise the lemma holds.  
Suppose  $|X'|=4$. Then both $(X'\cup e,Y'-e)$ and $(X',Y')$ are exact $4$-separations.  
Thus $e\in\cl ^*(X')\cap\cl ^*(Y'-e)$ or $e\in\cl (X')\cap \cl (Y'-e)$.  
The latter holds otherwise $M$ has a $4$-cocircuit; a \cn.   
But $Y'$ is coclosed, so $e$ was added by coclosure; that is, $e\in\cl ^*(Y-e)$ and we have a \cn\ to \ort\ since $e\in\cl (X)$.  
It remains to consider the case when $|X'|=3$.  
Then $|X'\cup e|=4$.  
The lemma holds if $X'\cup e=X$, so there is an element $f$ of $Y'-e$ that was added immediately before $e$ in the construction of $Y'$.  Now if $f$ is added via closure, then we can also add $e$ and $X'$ via closure, so we violate our choice of $Y'$. Thus $f$ is added via coclosure so 
$f \in \cl^*(Y' - e - f) \cap \cl^*(X' \cup e)$. Hence $M$ has a $4$-cocircuit; a \cn. 

We may now assume  that $X'$ is a cocircuit.  Then  $X'$ has at least six elements.  
As $X'$ is $4$-separating, $3= r(X')+r^*(X')-|X'|=r(X')-1$.  
Hence $r(X')=4$, so $M|X'$ is a restriction of $PG(3,2)$.  
As $M$ is binary, $X'$ contains no triangle and no $5$-circuits, so $M|X'$  is a restriction of $AG(3,2)$. As $X'$ has six or eight elements, it follows that $X'$ is a union of $4$-circuits so $\fcl(Y')$ cannot contain $X'$; a \cn. 
  \end{proof}

\begin{lemma}
\label{seqok}
Assume $M$ has no $4$-cocircuits. If every exact $4$-separation in $M$ is sequential, then, for every element $e\in E(M)$, the matroid $\si(M/e)$ is \ifc\ with no triads.  
\end{lemma}

\begin{proof}
Let $\{e,f_i,g_i\}$ be a triangle for all $i\in\{1,2,3\}$.  
The matroid $M'=\si (M/e)=M/e\ba f_1,f_2,f_3$ is \thc\ by Lemma~\ref{3conn}.  
Let $(U,V)$ be a \ftv\ in $M'$.  Then $|U|, |V| \ge 4$. 
Add $f_i$ to the side of the \ths\ containing $g_i$ for all $i\in\{1,2,3\}$ to obtain $(U',V')$, a \ths\ in $M/e$.  
Neither $(U'\cup e,V')$ nor $(U',V'\cup e)$ is a \ths\ in $M$.  
Hence both are $4$-separations in $M$.  
Thus, by hypothesis, each is a sequential $4$-separation  in $M$.  
Lemma~\ref{seqsep} implies that, without loss of generality, either $U'\cup e$ is coindependent and has rank at most three in $M$; or both $U'$ and $V'$ have rank at most three and are contained in $\cl (V'\cup e)$ and $\cl (U'\cup e)$, respectively.  
In the first case, as $U'\cup e$ is contained in a plane, $U$ is contained in a triangle in $\text{si}(M/e)$; a \cn.  
In the second case, $r(M)=4$, so $U'$ and $V'$ span planes in $PG(3,2)$.  
These planes meet in a line, so $|U'\cup V'|\leq 7+7-3=11$.  
Hence $E(M)\leq 12$; a \cn.  

Suppose $M/e\ba f_1,f_2,f_3$ has a triad $\{a,b,c\}$.  
Then, by Lemma~\ref{no5}, $M$ has $\{a,b,c,f_1,f_2,f_3\}$ as a cocircuit, so we may assume that $(a,b,c)=(g_1,g_2,g_3)$.  
Now $M$ has a triangle containing $f_1$ and exactly one of $f_2,g_2,f_3$, or $g_3$.  
It follows that $\si (M/e)$ has a triangle meeting $\{g_1,g_2,g_3\}$, so $\si (M/e)$ has a $4$-fan; a \cn.  
\end{proof}

\begin{figure}
\includegraphics[scale=1]{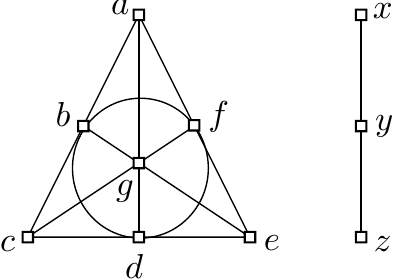}
\caption{A skew plane and line in a binary matroid.  Squares indicate positions that may be occupied by elements of $M$.}
\label{pl}
\end{figure}

The next three lemmas deal  with a plane and a line in $M$.  

\begin{lemma}
\label{baddie}
Suppose $M$ contains a plane $P$ and a line $L$ that are skew and are labelled as in Figure~\ref{pl} where not every element in the figure must be in $M$.  
If   $a,b,c,d,e,f,x,y$, and $z$ are  in $M$, and $\{x,y,a,b,d,e\}$ and $\{y,z,b,c,e,f\}$ are cocircuits in $M$, then $\si (M/w)$ is \ifc\ for all $w$ in 
$\{a,b,c,d,e,f\}$.  
\end{lemma}

\begin{proof}
By symmetric difference, $\{x,z,a,c,d,f\}$ is a cocircuit.  
As $z$ is in three triangles of $M$, \ort\ implies that $z$ is in a triangle with $c$, say $\{z,c,c'\}$, and a triangle with $f$, say $\{z,f,f'\}$.  Likewise, 
$x$ is in triangles $\{x,a,a'\}$ and $\{x,d,d'\}$, while $y$ is in triangles $\{y,b,b'\}$ and $\{y,e,e'\}$, for some elements $a',d',b',e'$. As $P$ and $L$ are skew, all of $a',b',c',d',e',f'$ are distinct and none is in $P$ or $L$. 

%Now $M'=\si (M/z)=M/z\ba c,f,y$ is \thc\ by Lemma~\ref{3conn} and has an $M(K_4)$-restriction  on $X = \{a,b,c',d,e,f'\}$.  
%Suppose that $(U,V)$ is a \ftv\ in $M'$.  
%Without loss of generality, $U$ spans  $X$ in $M'$. 
%Then $(U \cup X \cup \{c,f,y\}, V-X)$ is $3$-separating in $M/z$, so $(U \cup X \cup \{c,f,y\} \cup z, V-X)$ is $3$-separating in $M$. Since this is not a \ftv\ of $M$, it follows that $V$ is a sequential $3$-separating set in $M'$, so $V$ contains a triad $\{\beta,\gamma,\delta\}$. Then $M$ has either $\{\beta,\gamma,\delta, c,f,y\}$ as a cocircuit, or has $\{\beta,\gamma,\delta\} \cup s$ as a cocircuit for some $s$ in $\{c,f,y\}$. In the first case, by \ort\ with the circuits $\{z,x,y\}$, $\{z,c,c'\}$, and $\{z,f,f'\}$, we see that $\{\beta,\gamma,\delta\} = \{x,c',f'\}$. Then the cocircuit $\{x,c',f', c,f,y\}$ meets the triangle $\{a,b,c\}$ in a single element; a \cn. In the second case, the three triangles containing each choice for $s$ imply that $\{\beta,\gamma,\delta,s\}$ is $\{x,f',c',c\}$, $\{\mu_1,\nu_1,f',f\}$, or $\{\mu_2,\nu_2,x,y\}$ where $\{\mu_1,\nu_1\}$ meets $\{a,e\}$ and $\{b,d\}$, while $\{\mu_2,\nu_2\}$ meets each of $\{a,a'\}, \{d,d'\}, \{b,b'\}$, and $\{e,e'\}$. Hence $\{\beta,\gamma,\delta,s\} \neq \{\mu_2,\nu_2,x,y\}$. Moreover, $\{x,f',c',c\}$ and $\{\mu_1,\nu_1,f',f\}$ violate \ort\ with, respectively, $\{a,b,c\}$, or one of $\{x,a,a'\}$ and $\{y,e,e'\}$. We conclude that $\si(M/z)$ is \ifc. By symmetry, so are both $\si(M/x)$ and $\si(M/y)$.

By symmetry, it suffices to show that $\si(M/a)$ is \ifc. Let $M' = \si(M/a) = M/a\ba a',b,f$. Let $Z = \{c,d,e,x,y,z,d',b',f'\}$. The restriction of $M'$ to $Z$ is isomorphic to $M^*(K_{3,3})$. Suppose $(U,V)$ is a \ftv\ of $M'$. Without loss of generality, $U$ spans $Z$ in $M'$. 
Thus $U$ spans $\{c',e'\}$. Hence $(U\cup Z \cup \{c',e'\} \cup \{a',b,f\}, V - Z - \{c',e'\})$ is $3$-separating in $M/a$, so $(U\cup Z \cup \{c',e'\} \cup \{a',b,f\} \cup a, V - Z - \{c',e'\})$ is $3$-separating in $M$. Thus $V$ is a sequential $3$-separating set in $M'$, so $V$ contains a triad $\{\beta,\gamma,\delta\}$. Thus either $\{x,c,e, a',b,f\}$   or   $\{\beta,\gamma,\delta\} \cup t$ is a cocircuit of $M$ for some $t$ in $\{a',b,f\}$. The first possibility gives a \cn\ to \ort\ with $\{y,b,b'\}$. Thus $\{\beta,\gamma,\delta,b\}$, $\{\beta,\gamma,\delta,f\}$, or  $\{\beta,\gamma,\delta,a'\}$ is a cocircuit. Suppose $\{\beta,\gamma,\delta,b\}$ or $\{\beta,\gamma,\delta,f\}$ is a cocircuit. 
Then \ort\ implies that $\{\beta,\gamma,\delta\}$ contains $\{b,c,d\}$ or $\{f,e,d\}$ and so we get a \cn\ to \ort\ with  at least one of $\{x,d,d'\}, \{z,c,c'\}, \{z,f,f'\},  \{y,b,b'\}$ and $\{y,e,e'\}$. Thus $\{\beta,\gamma,\delta,a'\}$ is a cocircuit. This cocircuit also contains $x$ so either contains $y$ and elements from each of $\{b,b'\}$ and $\{e,e'\}$, or contains $z$ and elements from each of $\{f,f'\}$ and $\{c,c'\}$. Each case gives a \cn\ to \ort. We conclude that $\si(M/a)$ is \ifc, so the lemma holds.
\end{proof}

\begin{lemma}
\label{notfig}
Assume $M$ has no $4$-cocircuits. Let $(U,V)$ be a non-sequential $4$-separation of $M$ where $U$ is closed and $V$ is contained in the union of a plane $P$ and a line $L$ of $M$. Then  either $V$ is $6$-cocircuit, or $|V| = 9$ and $|P| = 6$. Moreover, 
$\si (M/v)$ is \ifc\ for at least six elements $v$ of $V$.  
\end{lemma}

\begin{proof}
By Lemma~\ref{noodd}, each cocircuit contained in $V$ has  exactly six elements otherwise it contains a triangle. 
Suppose $r(V)=3$.  
As $r(V)+r^*(V)-|V|=3$, we know that $V$ is coindependent.  
Hence it is contained in $\cl (U)$; a \cn.  
Evidently $r(V)\geq 4$.  
We use Figure~\ref{pl} as a guide for the points that may exist in $V$.  
We consider which positions are filled, keeping in mind that $V$ is the union of circuits and the union of cocircuits.  

Suppose $V$ has rank four and view $V$ as a restriction of $Q=PG(3,2)$.  
Then $\cl _Q(P)\cap \cl _Q(L)$ is a point of $Q$, so we may suppose $e=z$.  
Furthermore, as $r(V)+r^*(V)-|V|=3$, we know that $V$ contains, and therefore is, a cocircuit.  
Thus $|V|=6$.  
As $V$ contains no triangles, $|(P\cup L)\cap \cl _Q(P)|\leq 4$, and $|(P\cup L)\cap \cl _Q(L)|\leq 2$.  
Thus $e\notin P\cup L$.  
Without loss of generality, the points in $V$ are $a,b,f,g,x$, and $y$, and the result follows by Lemma~\ref{no6} provided $e\in E(M)$.  

We assume therefore that $e\notin E(M)$.  
We know that $V=\{x,y,a,b,f,g\}$.  
By \ort, without loss of generality, the three triangles of $M$ containing $x$ are $\{x,a,a'\},\{x,f,f'\}$, and $\{x,b,b'\}$.  
Thus $M$ has as triangles each of $\{y,a',f\},\{y,a,f'\}$, and $\{y,b',g\}$.  
Hence $M$ has no other triangles containing $x$ or $y$.  
Thus the remaining triangles containing $g$ must be in $P$, and so contain $c$ and $d$.  
But then $\{a,b,c\}$ and $\{a,g,d\}$ are triangles of $M$, so $a$ is in four triangles; a \cn.

Suppose that $r(V)=5$.  
Then $P$ and $L$ are skew, and $V$ is the union of two $6$-cocircuits, $C^*$ and $D^*$.   By \ort, each of $C^*$ and $D^*$ contains at most four elements of  $P$. Thus, by \ort, $|P| \le 6$ so $|C^* \cup D^*| \le 9$.  Hence $|C^* \btu D^*| = 6$ and $|V| = 9$. Then, without loss of generality, each of $C^*$ and $D^*$ meets $P$ in  four elements and $L$ in two elements. The result now follows    by Lemma~\ref{baddie}.  
  \end{proof}

\begin{lemma}
\label{nospike}
If $M$ has a $6$-element cocircuit $C^*=\{a,b,c,d,e,f\}$ where $\{a,b,c,d\}$ and $\{a,b,e,f\}$ are circuits, then   $\si (M/x)$ is \ifc\ for all $x$ in $C^*$.  
\end{lemma}

\begin{proof}
By symmetric difference, $\{c,d,e,f\}$ is also a circuit. Thus $C^*$ is the union of three disjoint pairs, $\{a,b\}$, $\{c,d\}$, and $\{e,f\}$ such that the union of any two of these pairs is a circuit. If one of these pairs is in a triangle with some element $x$, then each of the pairs is in a triangle with $x$ and the lemma follows by Lemma~\ref{no6}. Thus we may assume that each of $\{a,c\}$ and $\{a,d\}$ is in a triangle. Hence so are $\{b,c\}$ and $\{b,d\}$. Thus each of $a, b, c$ and $d$ is in exactly one triangle with an element of $\{e,f\}$. Hence $e$ and $f$ cannot both be in exactly three triangles; a \cn.
 \end{proof}
 
 \begin{lemma}
 \label{sixy}
 Let $(J,K)$ be an exact $4$-separation of $M$ such that $J$ is closed. If $|K| \le 6$, then $K$ is a  $6$-cocircuit and $\si(M/k)$ is \ifc\ for all $k$ in $K$. 
 \end{lemma}
 
 \begin{proof}
 We have $r(K) + r^*(K)  - |K| = 3$ and $|K| \ge 4$. If $|K| = 4$, then $K$ is a cocircuit; a \cn. Thus $|K| \ge 5$. Since $K$ is a union of cocircuits each of which has even cardinality, it follows that $|K| \ge 6$.  Hence $K$ is a $6$-cocircuit. Thus $r(K) = 4$ so $K$ contains two circuits such that they and their symmetric difference have even cardinality. Hence $K$ is the union of two 4-circuits that meet in exactly two elements and the result follows by Lemma~\ref{notfig}.
 \end{proof}

\section{The proof of the main result}

The next lemma essentially completes the proof of Theorem~\ref{bigun}. 

\begin{lemma}
\label{no4}
Let $M$ be an \ifc\ binary matroid in which every element is in exactly three triangles. Suppose $M$ has no $4$-cocircuits. Then $M$ has at least six   elements $e$ such that $\si(M/e)$ is \ifc.
\end{lemma}

\begin{proof} 
By Lemma~\ref{smalltime}, we know that $|E(M)| \ge 14$. Assume that the lemma fails. By Lemma~\ref{seqok}, $M$ has a non-sequential $4$-separation $(X,Y)$ where $X$ is minimal. Then $Y$ is fully closed. By Lemma~\ref{sixy},  $|X| \ge 7$ and $X$ contains an element $\alpha$ such that $\si(M/ \al)$ is not \ifc.  Let $\{\al ,f_i,g_i\}$ be a triangle for all $i\in\{1,2,3\}$.  
Now $M'=\si (M/\al)=M/\al \ba f_1,f_2,f_3$ is not \ifc.  
By Lemma~\ref{3conn}, it is \thc.  
Take a \ftv\ $(U',V')$ in $M'$. Then $|U'|, |V'| \ge 4$. Hence $r_{M/{\alpha}}(U')$ and $r_{M/{\alpha}}(V')$ exceed two. 
Add $f_i$ to the side containing $g_i$ for all $i\in\{1,2,3\}$ to obtain $(U'',V'')$.  
Then both $(U''\cup \al ,V'')$ and $(U'',V''\cup \al)$ are exact $4$-separations of $M$.  
Since $\alpha \in \cl(U'')$ and $\alpha \in \cl(V'')$, we deduce that $r_M(U'') \ge 4$ and $r_M(V'') \ge 4$. 
Moreover, by Lemma~\ref{seqsep}, both $(U''\cup \al ,V'')$ and $(U'',V''\cup \al)$ 
are non-sequential. 
Without loss of generality, we may assume that $r(U''\cap X)\geq r(V''\cap X)$ and, when equality holds, $|U''\cap X|\geq |V''\cap X|$. 
Let $(U,V) = (\cl(U''), V'' - \cl(U''))$. Then 

\begin{sublemma}
\label{bigger}
$r_M(U\cap X)\geq r_M(V\cap X)$, and, when equality holds, $|U\cap X|> |V\cap X|$.  
\end{sublemma}

We show next that 
\begin{sublemma}
\label{nonempty}
$X\cap U,X\cap V,Y\cap U$, and $Y\cap V$ are all non-empty.  
\end{sublemma}
As $\al\in X\cap U$, the first set is not empty.  
If the second is empty, then, as $\al$ is in the closure of $V=V\cap Y$, we can move $\al$ to $Y$ to get $(X-\al,Y\cup \al)$ as a \ns\ $4$-separation of $M$; a \cn\ to our choice of $(X,Y)$.  
If the third is empty, then $U=X\cap U$, and $(X\cap U,Y\cup V)$ contradicts   our choice of $(X,Y)$.  
Likewise, if the fourth set is empty, then $V=X\cap V$, and $(X\cap V,Y\cup U)$ violates our choice of $(X,Y)$.  
This completes our proof of~\ref{nonempty}.  

By submodularity of the connectivity function, $\lambda _M(X\cup U)+\lambda _M(X\cap U)\leq \lambda _M(X)+\lambda _M(U)=3+3$.  
We now break the rest of the argument into the following two cases, which we shall then consecutively eliminate.  
\begin{enumerate}
\item[(A)] $\lambda (X\cap U)\geq 4$ and $\lambda (X\cup U)=\lambda (Y\cap V)\leq 2$; or 
\item[(B)] $\lambda (X\cap U)\leq 3$.  
\end{enumerate}

\begin{sublemma}
\label{nope}
 (A) does not hold.  
\end{sublemma}

Suppose that (A) holds. As $M$ is \ifc, $Y\cap V$ is a triangle, or a triad, or contains at most two elements.  
Clearly, this set is not a triad.  
Suppose $\lambda (X\cap V)\geq 4$. Then, by submodularity again, $\lambda (Y\cap U)\leq 2$, so $|Y\cap U|\leq 3$.  
Then $|Y|\leq 6$, so $Y$ contains and so is a cocircuit. As this cocircuit cannot contain a triangle, it follows that $|Y \cap V| \le 2$, so $|Y| \le 5$; a \cn. Thus  $\lambda (X\cap V)\leq 3$.  
%By Lemma~\ref{notfig}, $V$ is not contained in the union of a plane and a line.  
If $\lambda (X\cap V)\leq 2$, then $X\cap V$ is contained in a triangle, so $V$ is contained in the union of two lines; a \cn\ since $V$ contains a cocircuit that must have six elements and so contain a triangle.  
We deduce that  $\lambda (X\cap V)=3$. Hence $X\cap V \subseteq \fcl(Y\cup U)$.  
Lemma~\ref{seqsep} implies that $X\cap V$ has rank at most three. Thus $V$ is contained in the union of a line $L$ and a plane $P$. It now follows by Lemma~\ref{sixy} that \ref{nope} holds.  

Next we show that 

\begin{sublemma}
\label{nopeB}
(B) does not hold.
\end{sublemma}

Assume that (B) holds. 
Since $\lambda (X\cap U)\leq 3$ and $X\cap U$ is properly contained in $X$, either $ X\cap U \subseteq \fcl(Y\cup V))$,  or $\lambda (X\cap U)\leq 2$.  
It follows using Lemma~\ref{seqsep} that $r(X\cap U)\leq 3$.  
Thus, by~\ref{bigger}, $r(X\cap V)\leq 3$.  
If $r(X\cap V)\leq 2$, then $X$ is contained in the union of a plane and a line. Then, arguing as in (A), it follows that $|X| = 6$ or $|X| = 9$ and $\si(M/x)$ is \ifc\ for all $x$ in $X$. Each alternative gives a contradiction. Thus,    
by~\ref{bigger}, $r(X\cap V)=3=r(X\cap U)$ and $|X\cap V|<|X\cap U|\leq 7$.  
Hence $4\leq r(X)\leq 6$.  

Now view $M$ as a restriction of $Q=PG(r-1,2)$, where $r=r(M)$.  
As $(X,Y)$ is an exact $4$-separation, $\cl _Q(X)\cap \cl _Q(Y)$ is a plane $P$ of $Q$.  
Because $Y$ is fully closed, no element of $X$ is in $P$.  
It follows by \ort, since $X$ is a union of cocircuits of $M$, that each triangle that meets an element of $X$ is either contained in $X$ or contains exactly two elements of $X$ with the third element being in $P$.  

We show that 
\begin{sublemma}
\label{5or6}
$r(X)\in\{5,6\}$.  
\end{sublemma}
Suppose not.  
Then $r(X)=4$ and $X\subseteq \cl _Q (X)-P$.  
So $X$ is contained in an $AG(3,2)$-restriction of $M$.  
As $X$ is a cocircuit, $|X|=6$ or $|X|=8$.  
Since $|X\cap U|\neq |X\cap V|$ and each is at least three, it follows that $|X|=8$.  
To have a triangle meeting $X$, there must be an element $y$ of $Y$ in $P$.  
But $y$ is the tip of a binary spike in $X\cup y$ so it is in at least four triangles.  
This \cn\ proves~\ref{5or6}.  

We show next that 
\begin{sublemma}
\label{only5}
$r(X)=5$.  
\end{sublemma}
Suppose not.  
Then $r(X)=6$.  
As $r(X\cap U)=r(X\cap V)=3$, we deduce that $\cl _Q(X\cap U)\cap \cl _Q (X\cap V)=\emptyset$, where we recall that $Q=PG(r-1,2)$ and $P=\cl _Q (X)\cap \cl _Q (Y)$.  

Suppose $\cl _Q(X\cap V)$ meets $P$.  
As $3=\lambda (X)=r(X)+r(Y)-r(M)$, we know that $r(Y)=r(M)-3$.  
Then $\cl _M(Y\cup (X\cap V))$ is a flat with rank at most $r(M)-1$.  
Hence its complement, which is contained in $X\cap U$, contains a cocircuit.  
But this cocircuit contains at least six elements by Lemma~\ref{no5}, so it contains a triangle in $X\cap U$.  
We deduce that $\cl _Q(X\cap V)$ avoids $P$.  
By symmetry, so does $\cl _Q(X\cap U)$.  
It follows that each triangle that meets $X$ is either contained in $X\cap U$ or $X\cap V$, or contains an element of each of $X\cap U,X\cap V$, and $P$.  
If $|X\cap U|=7$, then $M|(X\cap U) \cong F_7$-restriction, so each element in $X\cap U$ is in three triangles contained in $X\cap U$.  
Then each element in $X\cap V$ is contained in three triangles in $X\cap V$, so $M|(X \cap V) \cong F_7$,  and $|X\cap U|=|X\cap V|$; a \cn\ to~\ref{bigger}.  
Thus $|X\cap U|\leq 6$ and~\ref{bigger} implies that $|X\cap V|\leq 5$.  
Thus $X\cap V$ contains an element $v$ that is in at most one triangle in $X\cap V$.  
Hence $v$ is in triangles $\{v,u_1,p_1\}$ and $\{v,u_2,p_2\}$ for some $u_1$ and $u_2$ in $X\cap U$, and $p_1$ and $p_2$ in $P$.  
Take $u_3$ in $X\cap U$ such that $\{u_1,u_2,u_3\}$ is a basis for $X\cap U$.  
Then $\cl (Y\cup \{v,u_3\})$ is a flat of rank at most $r(M)-1$ whose complement, which is contained in $X\cap V$, contains a cocircuit.  
This cocircuit has at most five elements; a \cn\ to Lemma~\ref{no5}.  
Hence~\ref{only5} holds.

% Ending 1 of 2
We now know that $r(X)=5$.  
%Let $\cl _Q(X\cap U)$ be $P_1$ and $\cl _Q(X\cap V)$ be $P_2$, as shown in Figure~\ref{fig2}.  
It follows, since $r(X\cap U)=r(X\cap V)=3$, that $\cl _Q(X\cap U)\cap \cl _Q(X\cap V)$ is a point $p$ of $Q$.  
Moreover, $r(Y)=r(M)-2$, so $r(\cl _Q(Y)\cap \cl _Q(X\cap U)) =1$ since $r(\cl _Q(Y\cup (X\cap U)))=r(M)$, otherwise $X\cap V$ contains a cocircuit of $M$ that either has fewer than six elements or contains a triangle.  
Similarly, $r(\cl _Q(Y)\cap \cl _Q(X\cap V))=1$.  

The following is an immediate consequence of the fact that $U$ is closed.

\begin{sublemma}
\label{wheresp}
If $p \in X$, then $p \in X \cap U$.
\end{sublemma}

Let $\cl _Q(Y)\cap \cl _Q(X\cap U)=\{s\}$ and $\cl _Q(Y)\cap \cl _Q(X\cap V)=\{t\}$.  
Neither $s$ nor $t$ is in $X$, so 

$$|X\cap U|\leq 6.$$
  
Hence $|X\cap V|\leq |X\cap U|-1\leq 5$.  
Recall that $|X| \ge 9$. As $|X \cap U| \ge |X \cap V|$, it follows that $|X \cap U| \ge 5$. Hence

\begin{sublemma}
\label{cap56}
$|X\cap U|\in\{5,6\}$.  
\end{sublemma}

Call a triangle of $M$ \emph{special} if it contains an element of $X\cap U$, an element of $X\cap V$, and an element of $P$.  
Construct a bipartite graph $H$ with vertex classes $X\cap U$ and $X\cap V$ with $uv$ being an edge, where $u\in X\cap U$ and $v\in X\cap V$, precisely when $\{u,v\}$ is contained in a special triangle.  
Clearly 
\begin{equation}
\label{eq}
\sum _{u\in X\cap U} d_H(u)=\sum _{v\in X\cap V} d_H(v).
\end{equation}

Next we show the following.  
\begin{sublemma}
\label{star2}
Every vertex $x$ of $V(H)-\{ p\}$ has its degree in $\{1,2\}$.  
\end{sublemma}
Let $\{X',X''\}=\{X\cap U, X\cap V\}$ and take $x\in X'$ such that $x\neq p$.  
Let $x''$ be the element of $\cl _Q(X'')\cap P$.  
Thus $x''\in\{s,t\}$.  
Clearly $d_H(x)\leq 3$.  
Assume $d_H(x)=3$.  
Then $\cl _Q(Y\cup x)$ contains $x'$, at least three distinct elements of $X''$, and $x''$.  
Thus $\cl _Q(Y\cup x)$ contains $X''$.  
Hence $E(M)-\cl _M(Y\cup x)$ contains at most five elements of $M$; a \cn\ to the fact that every cocircuit of $M$ has at least six elements.  
Thus $d_H(x)<3$.  

Next suppose that $d_H(x)=0$.  
Then all three triangles containing $x$ are contained in $\cl _M(X')$.  
Thus $M|\cl _M(X')\cong F_7$.  
Hence, for $z\in X''-\cl _M(X')$, the three triangles containing $z$ are contained in $\cl _M(X'')$.  
Thus $M|\cl _M(X'')\cong F_7$.  
Hence $\cl _M(X')\cap \cl _M(X'')$ contains a point of $M$ that is in six triangles; a \cn.  
Thus~\ref{star2} holds.  

Now either 
\begin{enumerate}
\item[(i)] $s=t=p$; or 
\item[(ii)] $s,t$, and $p$ are distinct.  
\end{enumerate}

Suppose that (i) holds.  
Assume first that $p\notin Y$.  
By~\ref{star2}, for $W\in\{U,V\}$, every element of $M|(X\cap W)$ is in a triangle contained in $X\cap W$.  
Thus either $M|(X\cap W)\cong M(K_4)$ and  $\sum _{w\in X\cap W} d_H(w)=6$; or $M|(X\cap W)\cong M(K_4\ba e)$ and  $\sum _{w\in X\cap W} d_H(w)=9$.  
Since $|X\cap U|>|X\cap V|$, we obtain a \cn\ using (\ref{eq}).  
Thus $p\in Y$.  

As $|X\cap U|\in \{5,6\}$ by~\ref{cap56}, we see that $|X\cap U|=5$, otherwise $M|((X\cap U)\cup p)\cong F_7$, and $d_H(x)=0$ for every $x\in X\cap V$; a \cn\ to~\ref{star2}.  
We deduce that $M|((X\cap U)\cup p)\cong M(K_4)$, and $5=\sum _{u\in X\cap U} d_H(u)$.  
Now $p$ is in two triangles in $(X\cap U)\cup p$. Thus, of the three triangles in $\cl_Q(X \cap V)$ containing $p$, at most one contains two elements of $X \cap V$. Hence, using \ref{star2}, we see that  $M|\cl _M(X\cap V)$ comprises two triangles with a single element, not $p$, in common.  
Thus $\sum _{v\in X\cap V} d_H(v)=7$; a \cn\ to Equation~\ref{eq}.  
We conclude that (i) does not hold.  

We now know that $s,t$, and $p$ are distinct. 
We show next that 
\begin{sublemma}
\label{wheresp2}
  $p \in X$.
\end{sublemma}
 
Suppose $p\notin X$.  
Then $|X\cap U|=5$ so $|X\cap V|=4$.  
Thus $\sum _{u\in X\cap U} d_H(u)$ is five when $s\in Y$ and nine otherwise.  
As $d_H(v)<3$ for each $v\in X\cap V$ by~\ref{star2}, it follows that $t\in Y$.  
Then $\sum _{v\in X\cap V} d_H(v)$ is eight or seven depending on whether $M|(X\cap V)$ is $U_{3,4}$ or $U_{2,3}\oplus U_{1,1}$.  
Thus, by (\ref{eq}), we have a \cn.  
Hence \ref{wheresp2} holds.  

Suppose $|X\cap U|=6$.  
Then $s\notin Y$, otherwise there is an element of $(X\cap U)-p$ with degree zero in $H$; a \cn\ to~\ref{star2}.  
Then $\sum _{u\in X\cap U} d_H(u)=6$.  
Suppose $t\in Y$.  
If the line through $\{p,t\}$ contains a third point of $M$, say $q$, then each of the other two lines through $p$ in $\cl _Q(X\cap V)$ contains at most one point of $M$.  
Thus $|X\cap V|=3$ and, as $r(X\cap V)=3$, we see that $\{p,q,t\}$ is the unique triangle in $M|\cl _M(X\cap V)$ containing $q$.  
As this triangle is special, it follows that $d_H(q)=3$; a \cn\ to~\ref{star2}.  
  Evidently the line through $\{p,t\}$ does not contain a third point of $M$.  
We deduce that $M|\cl _M(X\cap V)$ comprises two triangles that have one element, not $p$ or $t$, in common.  
Then $\sum _{v\in X\cap V} d_H(v)=5$; a \cn.  
We deduce that $t\notin Y$.  
Then exactly one of the lines in $\cl _M(X\cap V)$ through $p$ contains exactly three points.  
Since no point of $X\cap V$  has degree three in $H$, it follows that $M|\cl _M(X\cap V)$ comprises two triangles with a point, not $p$, in common.  
As $p\notin X\cap V$, it follows that $\sum _{v\in X\cap V} d_H(v)=7$; a \cn.  
We conclude that $|X\cap U|\neq 6$.  

It remains to consider the case that $|X\cap U|=5$ and $|X\cap V|=4$.  
Then $\sum _{u\in X\cap U} d_H(u)$ is five or nine depending on whether or not $s$ is in $Y$.  
%Suppose $p\notin X$.  
%Then $t\in Y$, otherwise $d_H(v)=3$ for some $v\in X\cap V$; a \cn\ to~\ref{star2}.  
%Since $|X\cap V|=4$, we see that $M|\cl _M(X\cap V)$ consists of two triangles meeting in a point $z$.  
%It follows that $\sum _{v\in X\cap V} d_H(v)$ is eight if $z=t$ and seven otherwise; a \cn\ to Equation~\ref{eq}.  
From \ref{wheresp2}, $p\in X$.  
Thus  $M|\left[ (X\cap V)\cup p\right]$ consists of two three-point lines meeting in a point $z$.  
If $z=p$, then $\sum _{v\in X\cap V} d_H(v)$ is four or eight, depending on whether or not $t$ is in $Y$; a \cn.  
Hence $z\neq p$.  
Thus the third element on the line containing $\{p,t\}$ is in $X$.  
Again $\sum _{v\in X\cap V} d_H(v)$ is seven, if $t\notin Y$, or four, if $t\in Y$; a \cn\ to (\ref{eq}).   
We conclude that \ref{nopeB} holds and the lemma follows.
\end{proof}

It is now straightforward to complete the proof of our main result.

\begin{proof}[Proof of Theorem~\ref{bigun}.]
If $M$ has a $4$-cocircuit, then the result follows by Lemma~\ref{no5}. If $M$ has no $4$-cocircuits, then the theorem follows by Lemma~\ref{no4}.
\end{proof}

\section{A (non)-extension}

It is natural to ask whether, for an \ifc\ binary matroid $M$ with every element in exactly three triangles, $\si(M/e)$ is \ifc\ for every element $e$. We now describe an example where this is not the case. 

Begin with $K_{3,3}$ having vertex classes $\{a_1,a_2,a_3\}$ and $\{b_1,b_2,b_3\}$. Form the graph $G$ by adjoining three new vertices $u$, $v$, and $w$, each adjacent to all of $a_1,a_2,a_3,b_1,b_2$, and $b_3$ but not to each other. The vertex-edge incidence matrix of $G$ is the matrix $A$ shown below. 

{\SMALL
$$
\bordermatrix
{
&&&&&&&&&&&&&&&&&&&&&&&&&&&\cr
a_1 & 1 & 1 & 1 & 0 & 0 & 0 & 0 & 0 & 0 &1 & 0 &  0 & 0 & 0 & 0 & 1  & 0 &  0 & 0 & 0 & 0 & 1  & 0 &  0 & 0 & 0 & 0\cr
a_2 & 0 & 0 & 0& 1 & 1 & 1 & 0 & 0 & 0 &0  & 1 &  0 & 0 & 0 & 0 & 0  & 1 &  0 & 0 & 0 & 0 & 0  & 1 &  0 & 0 & 0 & 0\cr
a_3 & 0 & 0 & 0& 0 & 0 & 0 & 1 & 1 & 1 &0  & 0 &  1 & 0 & 0 & 0 & 0  & 0 &  1 & 0 & 0 & 0 & 0  & 0 &  1 & 0 & 0 & 0\cr
b_1 & 1 & 0 & 0& 1 & 0 & 0 & 1 & 0 & 0 &0  & 0 &  0 & 1 & 0 & 0 & 0  & 0 &  0 & 1 & 0 & 0 & 0  & 0 &  0 & 1 & 0 & 0\cr
b_2 & 0 & 1 & 0& 0 & 1 & 0 & 0 & 1 & 0 &0  & 0 &  0 & 0 & 1 & 0 & 0  & 0 &  0 & 0 & 1 & 0 & 0  & 0 &  0 & 0 & 1 & 0\cr
b_3 & 0 & 0 & 1& 0 & 0 & 1 & 0 & 0 & 1 &0  & 0 &  0 & 0 & 0 & 1 & 0  & 0 &  0 & 0 & 0 & 1 & 0  & 0 &  0 & 0 & 0 & 1\cr
u    & 0 & 0 & 0& 0 & 0 & 0 & 0 & 0 & 0 &1  & 1 &  1 & 1 & 1 & 1 & 0  & 0 &  0 & 0 & 0 & 0 & 0  & 0 &  0 & 0 & 0 & 0\cr
v    & 0 & 0 & 0& 0 & 0 & 0 & 0 & 0 & 0 &0  & 0 &  0 & 0 & 0 & 0 & 1  & 1 &  1 & 1 & 1 & 1 & 0  & 0 &  0 & 0 & 0 & 0\cr
w   & 0 & 0 & 0& 0 & 0 & 0 & 0 & 0 & 0 &0  & 0 &  0 & 0 & 0 & 0 & 0  & 0 &  0 & 0 & 0 & 0 & 1  & 1 &  1 & 1 & 1 & 1\cr
}
$$}
Then $M(G)$ is an \ifc\ matroid in which every element is in exactly three triangles. Now adjoin the matrix $B$ to $A$ where $B$ is shown below.
{\SMALL
$$
\bordermatrix
{
&a&b&c&d&e&f\cr
& 1 & 0 & 1 & 1 & 1 & 0 \cr
& 1 & 0 & 1 & 1 & 1 & 0 \cr
& 1 & 0 & 1 & 1 & 1 & 0 \cr
& 1 & 1 & 0& 0 & 1 & 1\cr
& 0 & 1& 1 & 1 & 0 & 1 \cr
& 0 & 1& 1 & 1 & 0 & 1 \cr
& 1 & 1 & 0& 0 & 1 & 1\cr
& 1 & 0 & 0& 1 & 0 & 1\cr
& 0 & 0 & 1& 0 & 1 & 1\cr
}
$$}

The matroid $N$ that is represented by $[A|B]$ has each element in $M(G)$ in exactly three triangles, and each element of $\{a,b,c,d,e,f\}$ is in exactly two triangles. To see this, observe that $N|\{a,b,c,d,e,f\} \cong M(K_4)$. Moreover, no element of $M(G)$ lies on a line with two elements of $\{a,b,c,d,e,f\}$ and it is straightforward to check that no element of $\{a,b,c,d,e,f\}$ is in a triangle with two elements of $M(G)$.

Now take the generalized parallel connection of $M(K_5)$ and $N$ across $\{a,b,c,d,e,f\}$ to get an \ifc\ binary matroid $M$ in which every element is in exactly three triangles. Evidently $\si(M/z)$ is not \ifc\ for all $z$ in $\{a,b,c,d,e,f\}$.

\end{document}